\documentclass[11pt,reqno]{amsart}

\usepackage{custom_preamble}

\begin{document}

\title{One-point concentration of the clique and chromatic numbers of the random Cayley graph on $\mathbb{F}_2^n$}

\author{\myname}
\address{\myaddress}
\email{\myemail}

\begin{abstract}
Green \cite{green} showed that there exist constants $C_1,C_2>0$ such that the clique number $\omega$ of the random Cayley graph on $\mathbb{F}_2^n$ satisfies $\lim_{n\to\infty}\mathbb{P}(C_1n\log n < \omega < C_2n\log n)=1$. In this paper we find the best possible $C_1$ and $C_2$. Moreover, we prove that for $n$ in a set of density $1$, clique number is actually concentrated on a single value. As a simple consequence of these results, we also prove the one-point concentration result for the chromatic number, thus proving the $\mathbb{F}_2^n$ analogue of the famous conjecture by Bollob\'{a}s \cite{bollobasconjecture} and giving almost the complete answer to the question by Green \cite{greenchromatic}.
\end{abstract}

\maketitle


\section{Introduction}

Clique number of a graph is the size of its largest complete subgraph. It is a classical result of the theory of random graphs that the clique number in the Erd{\H{o}}s-R\'{e}nyi model $G(N,\frac12)$ is strongly concentrated around $2\log N$ (all the logarithms appearing in this paper will be with base $2$). Even more so, Bollob{\'a}s and Erd{\H{o}}s \cite{bollobaserdos} and Matula \cite{matula} independently proved that the clique number is with high probability equal to one of the two values, and that for most $N$ it is actually concentrated on a single value. 

There is also a standard more algebraic way of constructing a random graph. Given an abelian group $G$, one constructs a random graph with vertex set $G$ by taking a random subset $A\subset G$ so that each each element lies in $A$ with probability $\frac12$, independently of other elements, and joining vertices $x\neq y$ if and only if $x+y\in A$. This graph is called \emph{random Cayley sum graph}. (There is also a related notion of \emph{random Cayley graph} where one takes $A$ to be a random \emph{symmetric} set and joins $x\neq y$ if and only if $x-y\in A$. In $\mathbb{F}_2^n$ these are the same, and for other groups most of the results true for one concept are also true for the other.) In \cite{green}, Green studied the clique number of the random Cayley sum graph on the groups $\mathbb{Z}/N\mathbb{Z}$ and $\mathbb{F}_2^n$. He proved that with high probability clique number is less than $160\log N$ in the $\mathbb{Z}/N\mathbb{Z}$ case and between $\frac14 n\log n$ and $308 n\log n$ in the $\mathbb{F}_2^n$ case. Green and Morris \cite{greenmorris} improved the 
former result by proving that for any $\epsilon>0$ the clique number is with high probability at most $(2+\epsilon)\log N$. The corresponding lower bound of $(2-\epsilon)\log N$ is much easier to prove, see \cite{greenchromatic}.

In this paper, we improve Green's result on $\mathbb{F}_2^n$. Let $\omega_n$ denote the corresponding clique number (we will omit subscript when it is clear from the context). The exact improvement is contained in the following theorem.

\begin{theorem}\label{T:main}
For every $\epsilon>0$,
$$\lim_{n\to\infty}\mathbb{P}(\textstyle{\frac12} n\log n < \omega_n <(1+\epsilon)n\log n)= 1.$$
\end{theorem}

Additionally, we also prove that the constants in front of $n\log n$ in the previous theorem are optimal, i.e.\ we prove the following theorem.

\begin{theorem}\label{T:optimalmain}
The constants $\frac12$ and $1+\epsilon$ in Theorem \ref{T:main} are optimal, i.e.\ there exist sequences $(n_i)$ and $(n_j)$ such that
$$\lim_{i\to\infty}\mathbb{P}(\textstyle{\frac12} n_i\log n_i< \omega_{n_i} < (\textstyle{\frac12} +\epsilon) n_i\log n_i)= 1$$
and
$$\lim_{j\to\infty}\mathbb{P}(n_j\log n_j < \omega_{n_j} < (1+\epsilon)n_j\log n_j)= 1.$$
\end{theorem}

Moreover, we prove that for most $n$ the clique number is concentrated on a single value. This is the $\mathbb{F}_2^n$ analogue of the celebrated two-point concentration result in the Erd{\H{o}}s-R\'{e}nyi model stated above. We note that this is the first result of that type for random Cayley sum graphs.

The word \emph{most} above is justified by the following corollary. For arbitrary $A\subset\mathbb{N}$ we define its density as $d(A)=\lim_{n\to\infty}|A\cap \{1,\dots,n\}|/n$ and its lower density by $\underline{d}(A)=\liminf_{n\to\infty}|A\cap \{1,\dots,n\}|/n$.

\begin{theorem}\label{T:onepoint}
There exists a sequence $(n_l)$ of density $1$ such that
$$\lim_{l\to\infty}\mathbb{P}(\omega_{n_l} = 2^{\lfloor \log n_l+\log\log n_l\rfloor})= 1.$$
\end{theorem}

In Section \ref{S:onepointchromatic} we will show that the analogous result is also true for the chromatic number $\chi$. The proof will turn out to be a straightforward consequence of the result about the clique number.

\begin{theorem}\label{T:onepointchrom}
There exists a sequence $(n_l)$ of density $1$ such that
$$\lim_{l\to\infty}\mathbb{P}(\chi_{n_l} = 2^{n-\lfloor \log n_l+\log\log n_l\rfloor})= 1.$$	
\end{theorem}

One can consider this result as the $\mathbb{F}_2^n$ analogue of the famous question by Bollob\'{a}s \cite{bollobasconjecture} claiming that the clique number in the Erd{\H{o}}s-R\'{e}nyi $G(N,\frac12)$ model is concentrated on $O(1)$ values which is still wide open. However, some very interesting results have been obtained in models $G(n,p)$ for $p$ very small, see e.g.\ the work of Alon and Krivelevich \cite{alonkrivelevich} where $p=n^{-1/2-\delta}$. 

Additionally, Green \cite{greenchromatic} asked what is, asymptotically, the chromatic number of the random Cayley graph of $\mathbb{F}_2^n$. The previous theorem gives a very precise answer for $n$ in a set of density $1$. In the same way we will prove Theorem~\ref{T:onepointchrom}, it is also possible to use other results for the clique number to say something about the chromatic number for every $n$, but unfortunately we were unable to give the complete answer to Green's question.

We will mostly concentrate on the upper bound in Theorem~\ref{T:main}, since the lower bound is essentially already contained in \cite{green}; we provide the details for the latter in Section~\ref{S:lower}. Notice that the set $X\subset \mathbb{F}_2^n$ spans a clique in the random Cayley graph generated by the set $A$ if and only if $X\widehat{+}X\subset A$, where $X\widehat{+}X = \{x_1+x_2\colon x_1,x_2\in X, x_1\neq x_2\}$. This observation, together with Markov's inequality, was used in \cite{green} to derive the inequality
$$\mathbb{P}(\omega\geq k)\leq \sum_{l\geq k-1}|S_k^l|2^{-l},$$
where $S_k^l$ is the family of sets $X\subset \mathbb{F}_2^n$ satisfying $|X|=k$ and $|X\widehat{+}X|=l$. If one has a good upper bound for $|S_k^l|$, the right hand side is $o(1)$ and $k$ is an upper bound for the clique number with probability tending to $1$. We will use this approach only to deal with sets $X$ of large doubling ($|X\widehat{+}X|\geq 10|X|$); we give details in Section~\ref{S:large}. The main downside of this approach is that it fails to deal with any dependencies between events $\{X\widehat{+}X\subset A\}$ for different sets $X$.

We have the obvious inequality
\begin{equation}\label{E:unionbound}
\mathbb{P}(\omega\geq k) \leq \mathbb{P}\left(\bigcup_{|X|=k, |X\widehat{+}X|<10k}\{X\widehat{+}X\subset A\}\right) + \mathbb{P}\left(\bigcup_{|X|=k, |X\widehat{+}X|\geq 10k}\{X\widehat{+}X\subset A\}\right).
\end{equation}

We denote the two terms on the right hand side by $P_{\mathrm{small}}(n)$ and $P_{\mathrm{large}}(n)$, respectively. In Section~\ref{S:small} we introduce a new approach which deals with some of the dependencies mentioned above, and thus effectively bounds $P_{\mathrm{small}}(n)$. We prove the optimality result in Section~\ref{S:optimality} and the one-point concentration of the clique number in Section~\ref{S:onepointclique}.

\textsl{Acknowledgements.} I would like to thank Ben Green for communicating the problem and helpful suggestions and to Sean Eberhard and Freddie Manners for fruitful discussions.

\section{Notation}\label{S:notation}

Although most of the notation was implicitly used in the introduction, we give it here for the reader's convenience. In this paper $N$ will always denote the size of the group $\mathbb{F}_2^n$, so $N=2^n$. For sets $X,Y\subset \mathbb{F}_2^n$ we will denote their sumset by $X+Y$, so $X+Y=\{x+y\colon x\in X, y\in Y\}$. We will also consider their \emph{restriced sumset} $X\widehat{+}Y=\{x+y\colon x\in X, y\in Y, x\neq y\}$. Notice that in the $\mathbb{F}_2^n$ setting we always have $X\widehat{+}Y=(X+Y)\setminus\{0\}$. We will also use standard $O$-notation: if $f,g$ are two functions on positive integers we will write $f(n)=O(g(n))$ (or $|f(n)|\lesssim |g(n)|$) if there exists $C>0$ such that $|f(n)|\leq C|g(n)|$ for all large enough $n$. Additionally, $f(n)=o(g(n))$ would mean that $f(n)/g(n)$ tends to $0$ as $n$ tends to infinity.

\section{Sets of large doubling}\label{S:large}

Our aim in this section is to prove that $P_{\mathrm{large}}(n)=o(1)$ for suitable choice of $k$. This is basically a consequence of results obtained in \cite{green} joined with a new result by Even-Zohar \cite{evenzohar}.

First of all, note that by the union bound, we have
$$\mathbb{P}\left(\bigcup_{|X|=k, |X\widehat{+}X|\geq 10k}\{X\widehat{+}X\subset A\}\right)\leq \sum_{l\geq 10k}|S_k^l|2^{-l}.$$

Let $X\subset \mathbb{F}_2^n$, $Y\subset \mathbb{F}_2^m$ be fixed sets. We will say that a bijection $f\colon X\to Y$ is a Fre\u{\i}man isomorphism if for every $x_1,x_2,x_3,x_4\in X$
\begin{equation}\label{E:freimandefinition}
x_1+x_2=x_3+x_4\text{ if and only if }f(x_1)+f(x_2)=f(x_3)+f(x_4).
\end{equation}
In this case we will say that sets $X$ and $Y$ are Fre\u{\i}man isomorphic. Let $r$ be the largest integer such that $X$ is Fre\u{\i}man isomorphic to a subset of $\mathbb{F}_2^r$ which is not contained in a proper affine subspace. This number $r$ is known as the Fre\u{\i}man dimension of the set $X$; we will denote it by $r(X)$. 

Proposition 26 from \cite{green} gives us the following bound for $|S_k^l|$.

\begin{proposition}\label{prop:bounds}
Let $r_{k,l} = \max_{X\in S_k^l}r(X)$. We have
$$|S_k^l|\leq N^{r_{k,l}+1}k^{4k}.$$
For $l\leq k^{31/30}$ there is a better bound
$$|S_k^l|\leq N^{r_{k,l}+1}\left(\frac{el}{k}\right)^{k}\exp (k^{31/32}).$$
\end{proposition}

We will bound $r_{k,l}$ from Proposition \ref{prop:bounds} using the following version of Fre\u{\i}man's theorem proved by Even-Zohar~\cite{evenzohar}.

\begin{theorem}\label{T:zohar}
Suppose $X\subset \mathbb{F}_2^n$ has cardinality $k$ and that $|X+X|\leq Kk$ for some $K\geq 1$. Then there is a subgroup $H$ containing $X$ such that
$$|H|\leq \frac{4^Kk}{2K}.$$
\end{theorem}

\begin{corollary}\label{C:zohar}
Suppose $X\subset \mathbb{F}_2^n$, $|X|=k$ and $|X+X|=l$. Then $$r(X)\leq \log k + 2l/k.$$
\end{corollary}
\begin{proof}
	Let $f\colon X\to\mathbb{F}_2^r$ be an injection which satisfies \eqref{E:freimandefinition}. Notice that $|f(X)+f(X)|=|X+X|\leq l$. By the previous theorem we conclude that $f(X)$ is contained in a subgroup of size at most $\frac{4^{K}k}{2K}$, for $K=l/k$. All subgroups of $\mathbb{F}_2^n$ are also subspaces and the claim easily follows.
\end{proof}

Let $k\geq \frac12 n\log n$. Notice that $|X+X|=|X\widehat{+}X|+1$ for any set $X$. Using the first bound in Proposition~\ref{prop:bounds} and the bound for the Fre\u{\i}man dimension given in Corollary~\ref{C:zohar} we get

\begin{align}
\sum_{l\geq k^{31/30}} |S_k^l|2^{-l} &\leq \sum_{l\geq k^{31/30}}N^{\log k + 2(l+1)/k+1}k^{4k}2^{-l} \leq \sum_{l\geq k^{31/30}}2^{l(-1+\frac{n\log k}{l}+\frac{3n}{k}+\frac{4k\log k}{l})}\nonumber\\
&\leq \sum_{l\geq k^{31/30}}2^{l(-1+o(1))} = o(1).\label{E:first}
\end{align}
Additionally, using the second bound in Proposition~\ref{prop:bounds}, we get
\begin{align}
\sum_{10k\leq l < k^{31/30}} |S_k^l|2^{-l} &\leq \sum_{10k\leq l < k^{31/30}}N^{\log k + 2(l+1)/k+1}\left(\frac{el}{k}\right)^{k}\exp (k^{31/32})2^{-l}\nonumber\\
&\leq \sum_{10k\leq l < k^{31/30}}2^{l(-1+\frac{n\log k}{l}+\frac{3n}{k}+\frac{k}{l}\log\left(\frac{el}{k}\right)+\frac{2k^{31/32}}{l})}\nonumber\\
&\leq \sum_{10k\leq l < k^{31/30}}2^{l(-1+\frac{1}{5}+\frac{k}{l}\log\left(\frac{el}{k}\right)+o(1))}=o(1).\label{E:second}
\end{align}
Combining \eqref{E:first} and \eqref{E:second} we conclude that $P_{\mathrm{large}}(n)=o(1)$ for any $k \geq \frac12 n\log n$.

\section{Sets of small doubling}\label{S:small}

In this section we deal with sets of small doubling. The strategy is the following. First of all, notice that every set $X$ of size $k$ satisfying $|X\widehat{+}X|<10k$ is by Theorem~\ref{T:zohar} contained in a subspace of size $O(k)$. Now fix a subspace $V$ of this size and consider all sets of size $k$ contained in it. The crucial claim we will prove here is that there exists a small family $\mathcal{F}$ of large sets such that for each $X$, the set $X\widehat{+}X$ contains a set from this family. This will enable us to exploit dependencies between events $\{X\widehat{+}X\subset A\}$ for various sets $X$. More precisely, we will use the following obvious inequality
\begin{align}
\mathbb{P}\left(\bigcup_{|X|=k, X\subset V}\{X\widehat{+}X\subset A\}\right) &= \mathbb{P}\left(\bigcup_{F\in\mathcal{F}}\bigcup_{|X|=k, X\subset V, F\subset X\widehat{+}X}\{X\widehat{+}X\subset A\}\right)\nonumber\\
&\leq \mathbb{P}\left(\bigcup_{F\in\mathcal{F}}\{F\subset A\}\right)\nonumber\\
&\leq \sum_{F\in\mathcal{F}}\mathbb{P}\left(F\subset A\right).\label{E:familyF}
\end{align}

Obviously, the crucial step in this approach is to find the family $\mathcal{F}$ with these properties and this is what we do in the following theorem.

\begin{theorem}\label{T:familyF}
Fix $\epsilon>0$. For large enough $n\geq n_0(\epsilon)$ the following holds. Let $k\geq \epsilon N$ and $k'$ be the power of $2$ satisfying $k'<k\leq 2k'$. There is a family $\mathcal{F}$ of subsets of $\mathbb{F}_2^n$ with the following properties 
\begin{enumerate}
\item[(i)] $|\mathcal{F}|\leq 2^{\epsilon N}$,
\item[(ii)] every $F\in \mathcal{F}$ has at least $(2-\epsilon)k'$ elements,
\item[(iii)] for every $X\subset \mathbb{F}_2^n$ with $k$ elements, there is $F\in \mathcal{F}$ such that $F\subset X+X$.
\end{enumerate}
\end{theorem}

A careful reader might object that $k\geq \epsilon N$ is not the setting we have been working with until now, since $k$ of our interest is of about logarithmic size of the size of the ambient group. However, in the application of this theorem, we will consider our $k$-element subset as being contained in some subgroup of size $O(k)$, which we can do by Theorem \ref{T:zohar}. So $N$ here actually represents the size of this subgroup.

There are couple of ingredients we need in the proof of this theorem. The first one is a regularity lemma in $\mathbb{F}_2^n$. The variant we will use was first proven by Green in \cite{greenregular}.

Let $V\leqslant \mathbb{F}_2^n$. We give $V$ the uniform probability measure and $V^{\ast}$ the counting measure. Hence, for any $f\colon V\to \mathbb{R}$ and $r\in V^{\ast}$ we define corresponding Fourier coefficient by
$$\widehat{f}(r)=\mathbb{E}_{v\in V}f(v)(-1)^{\langle v,r\rangle}.$$

Fix a subset $X\subset \mathbb{F}_2^n$. For $i\in \mathbb{F}_2^n$ we consider function $1_{(X-i)\cap V}\colon V\to \mathbb{R}$. We will say that $i$ is an $\epsilon$-regular value for $X$ with respect to $V$ if
$$\sup_{r\neq 0}|\widehat{1_{(X-i)\cap V}}(r)|\leq \epsilon.$$
Notice that $i$ is $\epsilon$-regular value if and only if all the elements in $i+V$ are $\epsilon$-regular values. Indeed, this follows easily by noticing that $|\widehat{1_{(X-i)\cap V}}(r)| = |\widehat{1_{(X-(i+v))\cap V}}(r)|$ for all $v\in V$. Thus, we may also talk about $\epsilon$-regular cosets of $V$. Additionally, we will say that $X$ is $\epsilon$-regular with respect to $V$ if there are less than $\epsilon N$ choices of $i$ which are not $\epsilon$-regular values for $X$ with respect to $V$.

The following is Theorem 2.1 from \cite{greenregular}.

\begin{theorem}[Regularity lemma in $\mathbb{F}_2^n$]\label{T:regularity}
 Let $\epsilon\in(0,\frac12)$ and $X\subset\mathbb{F}_2^n$. Then there is a subspace $V\leqslant\mathbb{F}_2^n$ of codimension $O_{\epsilon}(1)$ such that $X$ is $\epsilon$-regular with respect to $V$.
\end{theorem}

This result is the main ingredient in the proof of the following proposition which is all we will ultimately need.

\begin{proposition}\label{P:regularityapplied}
 Let $X\subset\mathbb{F}_2^n$, $\epsilon>0$ and $|X|\geq \epsilon N$. There is a subspace $V\leqslant\mathbb{F}_2^n$ of codimension $O_{\epsilon}(1)$ and $I\subset \mathbb{F}_2^n/V$ such that $|X\cap\bigcup_{i\in I}(i+V)|\geq (1-\epsilon)|X|$ and $|(X\cap (i+V))+(X\cap (i'+V))|\geq (1-\epsilon)|V|$ for all $i,i'\in I$.
\end{proposition}

\begin{proof}
 It is enough to prove the claim only for small enough $\epsilon$. Apply the regularity lemma with parameter $\epsilon^5$. We get at most $\epsilon^5N$ elements which are not $\epsilon^5$-regular values for $X$. By the comment above, this is equivalent to stating that at most $\epsilon^5|\mathbb{F}_2^n/V|$ cosets of $V$ are not $\epsilon^5$-regular. Let $I$ be the set of all translates $i+V$ of $V$ such that $|X\cap (i+V)|\geq \epsilon^3 |V|$ and $i$ is $\epsilon^5$-regular value for $X$ with respect to $V$. It is easy to see that $|X\cap\bigcup_{i\in I}(i+V)|\geq |X|-\epsilon^3N-\epsilon^5|\mathbb{F}_2^n/V||V|\geq (1-\epsilon)|X|$. 
 
 Fix $i,i'\in I$ and let $E=(X-i)\cap V$ and $F=(X-i')\cap V$. To prove the remaining property, it is enough to prove that $|E+F|\geq (1-\epsilon)|V|$.
 
Let $S=V\setminus(E+F)$. By Plancherel's formula
$$0=\mathbb{E}_{v\in V}1_E\ast 1_F(v)1_S(v) = |E||F||S|/|V|^3 + \sum_{r\in V^{\ast}\setminus\{0\}}\widehat{1_E}(r)\widehat{1_F}(r)\widehat{1_S}(r),$$
and hence
\begin{align*}
|E||F||S|/|V|^3 &\leq \epsilon^5\sum_{r\in V^{\ast}\setminus\{0\}}|\widehat{1_E}(r)|^{1/2}|\widehat{1_F}(r)|^{1/2}|\widehat{1_S}(r)|\\
&\leq \epsilon^5 \left(\sum_{r\in V^{\ast}}|\widehat{1_E}(r)|^2\right)^{1/4} \left(\sum_{r\in V^{\ast}}|\widehat{1_F}(r)|^2\right)^{1/4} \left(\sum_{r\in V^{\ast}}|\widehat{1_S}(r)|^2\right)^{1/2}\\
&\leq \epsilon^5 |E|^{1/4} |F|^{1/4} |S|^{1/2}/|V|.\\
\end{align*}
The conclusion is that $|S| \leq \epsilon^{10} |V|^4 |E|^{-3/2} |F|^{-3/2}$, which together with hypothesis on the sizes of sets $E$ and $F$ gives us what we want.
\end{proof}

The second major ingredient we will need in the proof of Theorem~\ref{T:familyF} is the following classical theorem by Kneser \cite{kneser}.

\begin{theorem}[Kneser]
 Let $G$ be an abelian group. For any set $S\subset G$ define $\mathrm{Sym}(S)=\{g\in G\colon g+S=S\}$. Then
 $$|A+B|\geq |A|+|B|-|\mathrm{Sym}(A+B)|.$$
\end{theorem}

Readers interested in the proof could consult the book by Tao and Vu \cite{taovu}. We use this theorem to derive the following simple consequence.

\begin{proposition}\label{P:sandwich}
Let $A,B\subset\mathbb{F}_2^n$ and suppose that $|B|>m$, where $m$ is some power of $2$. Then
$$|A+B|\geq \min(|A|+m,2m).$$
\end{proposition}
\begin{proof}
Let $W=\mathrm{Sym}(A+B)$. Obviously, $W$ is a subgroup, hence $|W|$ is a power of $2$. Notice that for any nonempty set $S$ we must have $|S|\geq |\mathrm{Sym}(S)|$. Indeed, for fixed $s\in S$ and $g\in\mathrm{Sym}(S)$, we must have that $t=g+s\in S$. However, $s$ and $t$ uniquely determine $g$, and since there are $|S|$ options for $t$, we are done. Thus, if $|W|\geq 2m$, then the proof is finished.

Suppose now that $|W|\leq m$ and let $\pi\colon \mathbb{F}_2^n\to \mathbb{F}_2^n/W$ be the canonical projection. Notice that $|\pi(A)|\geq |A|/|W|$ and $|\pi(B)|\geq m/|W|+1$. Kneser's theorem now gives us that $|\pi(A)+\pi(B)|\geq |\pi(A)|+|\pi(B)|-1$. It is easy to see that $A+B$ is equal to some union of cosets of $W$ and hence $|A+B|=|\pi(A+B)||W| = |\pi(A)+\pi(B)||W|$. Using this and the previous three inequalities we get what we wanted.
\end{proof}

We proceed to the proof of the main theorem in this section.

\begin{proof}[Proof of Theorem~\ref{T:familyF}]
In this proof we will use Proposition~\ref{P:regularityapplied} and hence will also use the notation introduced there. Obviously, it is enough to prove the claim only for small enough $\epsilon$.

Let $\eta=\epsilon^7$, and $d$ be the maximal codimension given by the regularity lemma when applied with parameter $\eta$. We will say that a set is \emph{almost coset of $V$} if it can be obtained by removing at most $\epsilon^3|V|$ elements from a coset of $V$. Our family $\mathcal{F}$ will consist of all subsets of $\mathbb{F}_2^n$ with at least $(2-\epsilon)k'$ elements that can be obtained as a union of almost cosets of some subspace $V$ of codimension at most $d$. Family $\mathcal{F}$ by its definition satisfies property (ii) from the statement of the theorem. Let's prove that it also satisfies property (i), that is $|\mathcal{F}|\leq 2^{\epsilon N}$ (for $n$ large enough).

Fix codimension $d'\leq d$ and let $D'=2^{d'}$. There are at most $N^{d'}$ ways to choose subspace $V$ of codimension $d'$, and $2^{D'}$ ways to choose some of its cosets from which we will remove some elements to obtain a set in $\mathcal{F}$. So the contribution from these two factors is bounded by $2^{o(N)}$. Once we have chosen cosets, we need to bound the number of ways to remove at most $\epsilon^3 N/D'$ elements from each of them. We bound this by
$$\left(\sum_{s=0}^{\epsilon^3N/D'}{N/D'\choose s}\right)^{D'}\leq  N^{D'} {N/D'\choose \epsilon^3N/D'}^{D'}\leq N^{D'} \left(e/\epsilon^3\right)^{\epsilon^3N}\leq 2^{o(N)}2^{\epsilon^2N}.$$
Multiplying these bounds, and taking the union over all the possible choices for $d'$ gives the required bound.

Let's prove now the remaining (and most important) property (iii). Take any $X\subset \mathbb{F}_2^n$, $|X|=k$, and apply Proposition~\ref{P:regularityapplied} with parameter $\eta$. Let $X_I=X\cap\bigcup_{i\in I}(i+V)$. 

Let $m$ be the largest power of $2$ such that $|I|>m/2$. Then obviously $|I|\leq m$ and by Proposition~\ref{P:sandwich} also $m\leq |I+I|$.

Let $\delta$ be the density of $X$ on the union of translates from $I$. Hence
$$(1-\eta)k'<(1-\eta)|X|\leq |X_I|=\delta|I||V| \leq \delta m|V|.$$
Notice that if $|X_I+X_I|\geq (2-\epsilon)k'$, then $X_I+X_I\in \mathcal{F}$ and we are done. If not, then
$$(2-\epsilon)k' > |X_I+X_I|\geq (1-\eta)|I+I||V|\geq (1-\eta)m|V|.$$
So
$$\frac{1-\eta}{\delta}< \frac{m|V|}{k'}< \frac{2-\epsilon}{1-\eta}<2.$$
However, since the second term is a power of $2$, we conclude that $\delta \geq 1-\eta$. 

Let $J$ be the set of all those $j\in I$ for which $|X\cap (j+V)|\geq (1-\sqrt{\eta})|V|$. By bounding the size of $X\cap\bigcup_{i\in I}(i+V)$ from below and from above we get
$$(1-\eta)|I||V| \leq |J||V| + (1-\sqrt{\eta})(|I|-|J|)|V|,$$
and hence $|J|\geq (1-\sqrt{\eta})|I|$.

If we set $X_J=X\cap\bigcup_{j\in J}(j+V)$, then
$$|X_J|\geq |X_I| - \sqrt{\eta} |I||V|\geq (1-\eta)|X|-\sqrt{\eta}N\geq(1-\eta-\sqrt{\eta}/\epsilon)|X|>(1-\epsilon)k'.$$
Proposition~\ref{P:sandwich} now shows that $|X_J+X|\geq (2-\epsilon)k'$, so $X_J+X\in \mathcal{F}$ and we are again done.
\end{proof}

By applying Theorem~\ref{T:familyF} for, say, $\epsilon/2$ and removing the element $0$ from all sets in $\mathcal{F}$, we get the following corollary.

\begin{corollary}\label{C:familyF}
 Theorem~\ref{T:familyF} is also true if we replace sumsets with restriced sumsets.
\end{corollary}

\section{Proof of the upper bound in Theorem \ref{T:main}}\label{S:upper}

We can now easily finish the proof of the upper bound in Theorem~\ref{T:main}. Let $k=\lfloor(1+\epsilon)n\log n\rfloor$ and $\delta>0$ to be specified later. As we said at the beginning of this section, any $X$ such that $|X|=k$ and $|X+X|< 10|X|$ is contained in some subspace of size $O(k)$. Since there are at most $N^{\log O(k)}$ such subspaces, applying Corollary~\ref{C:familyF} with parameter $\delta$ and from \eqref{E:familyF} we have
\begin{align}
P_{\mathrm{small}}(n) &\leq N^{\log O(k)}2^{\delta O(k)}2^{-(2-\delta)k'}\nonumber\\
&\leq N^{\log O(k)}2^{\delta O(k)}2^{-(1-\delta/2)k}\nonumber\\
&\leq 2^{n\log n(1 + \delta O(1)-(1-\delta/2)(1+\epsilon)+o(1))}\label{E:smallfinal}
\end{align}
and this is $o(1)$ if we choose $\delta$ sufficiently small, which, joined with the appropriate result for $P_{\mathrm{large}}(n)$ from Section~\ref{S:large} proves the upper bound.

\section{Proof of the lower bound in Theorem \ref{T:main}}\label{S:lower}

In this section we prove the lower bound stated in Theorem~\ref{T:main}. As we mentioned in the introduction, this is basically already contained in Green's initial paper \cite{green}, so here we just sketch his argument and do a more careful calculation.

Let $m=\lfloor\log n+\log\log n\rfloor$ and let $M$ be the number of $m$-dimensional subspaces $H$ of $\mathbb{F}_2^n$ satisfying $H\setminus \{0\}\subseteq A$. In \cite{green} it was proven that
$$\mathbb{E}M\geq 2^{nm-m^2-2^m}\quad\text{and}\quad\mathrm{Var}(M)\leq 2\sum_{l=1}^m2^{2mn-2^{m+1}+2^l-nl}.$$
Notice that $2^l-nl\leq 2-n$ for every $l=1,\dots, m$. Indeed, the crucial point is when $l=m$; in that case we have
\begin{align}
2^m-n(m-1)-2&\leq 2^{\log n+\log\log n}-n(\log n+\log\log n-2)-2\nonumber\\
&\leq -n\log\log n+2n-2 \leq 0.\label{eq:kn}
\end{align}
Second moment method now gives
$$\mathbb{P}(M=0)\leq \frac{\mathrm{Var}(M)}{(\mathbb{E}M)^2}\leq 8m2^{2m^2-n}=o(1).$$
Notice now that whenever $H\setminus\{0\}\subset A$ there is a clique of size at least $|H|$, since $H\setminus\{0\} = H\widehat{+}H$, so we get the lower bound
\begin{align}
\mathbb{P}(\omega > \textstyle{\frac12} n\log n) \geq\mathbb{P}(\omega\geq 2^{\lfloor\log n+\log\log n\rfloor})=1-o(1).\label{eq:low}
\end{align}

\section{Optimality}\label{S:optimality}

Finally, we prove that the results obtained in the previous two sections are optimal (in terms of constants in front of $n\log n$).

\begin{proof}[Proof of Theorem~\ref{T:optimalmain}]
First of all, let's prove that for every $\epsilon>0$, there is a sequence $(n_i)$ for which $k=\lfloor(\frac12+\epsilon)n\log n\rfloor$ is the upper bound for the clique number. We already know from Section~\ref{S:large} that $P_{\mathrm{large}}(n_i)=o(1)$.

Let $m_i=\left\lfloor \frac{1}{1+\epsilon}2^i\right\rfloor$ and $n_i = 2^{m_i}$. Notice that for such $n_i$ we have 
$$\left(1+\frac{\epsilon}{1+\epsilon}-(\textstyle{\frac12}+\epsilon)2^{1-i}\right)2^{m_i+i-1}-1\leq k\leq \left(1+\frac{\epsilon}{1+\epsilon}\right)2^{m_i+i-1}.$$
We remind the reader that we defined $k'$ as the power of $2$ satisfying $k'< k\leq 2k'$. By the previous two inequalities, we have $k' = 2^{m_i+i-1}$. Moreover, by the second inequality we have $(\frac12+\epsilon)k'/k\geq\frac12+\frac{\epsilon}{2}$. Following the same calculation as in \eqref{E:smallfinal} we now have
\begin{align*}
P_{\mathrm{small}}(n_i) &\leq 2^{n_i\log n_i(1 +\delta O(1)+o(1))-(2-\delta)k'}\leq 2^{n_i\log n_i(1 +\delta O(1)+o(1)-(2-\delta)(\frac12+\epsilon)k'/k)}\\
&\leq 2^{n_i\log n_i(1 +\delta O(1)+o(1)-(2-\delta)(\frac12+\frac{\epsilon}{2}))}
\end{align*}
and this is again $o(1)$ if we choose $\delta$ sufficiently small. Hence
$$\lim_{i\to\infty}\mathbb{P}(\omega\leq (\textstyle{\frac12} +\epsilon) n_i\log n_i)= 1$$
and the lower bound obtained in the previous section is optimal.

To prove that the upper bound from Section~\ref{S:small} is also optimal, first notice that the same calculations we did in Section~\ref{S:lower} would, for some $n$, go through for $m=\lfloor\log n+\log\log n\rfloor +1$. Indeed, it is easy to see that $2^{1-\{\log x\}}-1\lesssim \frac{1}{x}\leq \frac{\log x}{2x}$ for all large enough integers $x$ of the form $x=2^s-1$, so if we set $n_j=2^{2^j-1}$ we get that \eqref{eq:kn} is now
\begin{align*}
2^m-n_j(m-1)-2 &< (2^{1-\{\log\log n_j\}} - 1)n_j\log n_j-n_j\log\log n_j + n_j\\
&\leq \textstyle{\frac{\log\log n_j}{2\log n_j}} \cdot n_j\log n_j-n_j\log\log n_j + n_j = -\textstyle{\frac12}n_j\log\log n_j+n_j <0.
\end{align*}
Proceeding as before, this gives
\begin{equation}\label{E:onepointneoptimalno}
\mathbb{P}(\omega > n_j\log n_j)\geq \mathbb{P}(\omega\geq 2^{\lfloor\log n_j+\log\log n_j\rfloor +1})=1-o(1).
\end{equation}
\end{proof}

\section{One-point concentration of the clique number}\label{S:onepointclique}

In this section we prove that we actually have one-point concentration for most $n$, namely Theorem \ref{T:onepoint}. This will be an easy consequence of the following proposition.

\begin{proposition}\label{P:onepoint}
Let $\epsilon>0$ and $(n_l)$ be a sequence satisfying $\{\log n_l+\log\log n_l\} <1-\epsilon$. Then
$$\lim_{l\to\infty}\mathbb{P}(\omega_{n_l} = 2^{\lfloor \log n_l+\log\log n_l\rfloor})= 1.$$
\end{proposition}

\begin{proof}
Let $\epsilon >0$ and $(n_l)$ be a sequence as in the statement of the proposition. To prove the upper bound, we set $k=2^{\lfloor\log n + \log\log n\rfloor} +1$. As before, by the results of Section \ref{S:large} we already have $P_{\mathrm{large}}(n_l)=o(1)$. On the other hand, after applying Corollary~\ref{C:familyF} with parameter $\delta$ and by the same calculation as in \eqref{E:smallfinal} we get
$$P_{\mathrm{small}}(n_l) \leq 2^{n_l\log n_l(1+o(1)+\delta O(1)-(2-\delta)2^{-\{\log n_l+\log\log n_l\}})}.$$
Choosing $\delta$ small enough gives $P_{\mathrm{small}}(n_l)=o(1)$ and this proves that with high probability the clique number is at most $2^{\lfloor\log n + \log\log n\rfloor}$. Additionally, from \eqref{eq:low} we see that the corresponding lower bound is also true, which finishes the proof.
\end{proof}

We can now move to the proof of Theorem~\ref{T:onepoint}.

\begin{proof}[Proof of Theorem~\ref{T:onepoint}]
 First we prove that for any $\epsilon>0$ there exists a set $T_{\epsilon}$ such that $\underline{d}(T_{\epsilon})>1-\epsilon$ and
 $$\lim_{n\to\infty, n\in T_{\epsilon}}\mathbb{P}(w=2^{\lfloor \log n+\log\log n\rfloor})=1.$$
 To do this, by Proposition~\ref{P:onepoint}, it is enough to prove that for any $\epsilon >0$, the set of integers satisfying $\{\log n+\log\log n\} < 1-\epsilon/24$ has lower density at least $1-\epsilon$.
 
 Notice that within range $2^m\leq n< 2^{m+1}$, the quantity $\log\log n$ varies within an interval of size about $1/m$, so for $m$ large enough having $\{\log n\}\in (\epsilon/12,1-\epsilon/12)$ ensures that $\{\log n+\log\log n\} < 1-\epsilon/24$. However, there are $(2^{1-\epsilon/12}-2^{\epsilon/12})2^m > (1-\epsilon/3)2^m$ such $n$ in this range, and it is easy to see that this implies that the lower density is at least $1-\epsilon$.
 
 Let $S_{\epsilon}=\{n\in T_{\epsilon}\colon \mathbb{P}(w=2^{\lfloor \log n+\log\log n\rfloor})>1-\epsilon\}$. Obviously, $\underline{d}(S_{\epsilon}) \geq \underline{d}(T_{\epsilon})>1-\epsilon$. It is now easy to see that the sequence $\bigcup_nS_{1/n}$ satisfies all the requirements needed in the statement.
\end{proof}

Unfortunately, we cannot omit the condition that $n$ goes over a set of density $1$ and let $n$ go over all the integers. Indeed by \eqref{E:onepointneoptimalno} there exists a sequence $(n_j)$ such that
$$\lim_{j\to\infty} \mathbb{P}(w\geq 2^{\lfloor\log n_j+\log\log n_j\rfloor + 1})=1.$$

\section{One-point concentration of the chromatic number}\label{S:onepointchromatic}

In this final section we prove the one-point concentration result for the chromatic number, that is Theorem~\ref{T:onepointchrom}. As we will see, this is a very easy consequence of the corresponding result for the clique number - Theorem~\ref{T:onepoint} and its proof.

\begin{proof}[Proof of Theorem~\ref{T:onepointchrom}]
	First of all, notice that, by symmetry, all the results in previous sections would also hold if we considered the independence number (that is, the size of the largest independent set of vertices) instead of the clique number. By Theorem~\ref{T:onepoint} the independence number is with high probability equal to $2^{\lfloor \log n+\log\log n\rfloor}$, and hence, since every colour forms an independent set, by the pigeonhole principle the chromatic number is at least $2^{n - \lfloor \log n+\log\log n\rfloor}$. On the other hand, the proof of Theorem~\ref{T:onepoint} reveals that not only that the clique number is as claimed, but also that the largest clique is formed by a subspace $V$ of dimension $\lfloor \log n+\log\log n\rfloor$. Notice that if a subspace forms an independent set, then the same is true for each of its translates, since $(i+V)\widehat{+}(i+V)=V\widehat{+}V$ for any $i\in\mathbb{F}_2^n$. Hence we can define a valid colouring by colouring each translate of $V$ with its own colour, thus using exactly $2^{n - \lfloor \log n+\log\log n\rfloor}$ colours.
\end{proof}

\bibliography{clique}{}
\bibliographystyle{alpha}

\end{document}